\documentclass{amsart}
\usepackage{graphicx}

\newtheorem{theorem}{Theorem}[section]
\newtheorem{lemma}[theorem]{Lemma}
\newtheorem{prop}[theorem]{Proposition}
\newtheorem{cor}[theorem]{Corollary}

\newtheorem{ques}[theorem]{Question}
\newtheorem{rmk}[theorem]{Remark}

\newcommand{\Gab}{G  -  a,b}

\newcommand{\R}{\mathbb R}

\newcommand{\TY}{\nabla\mathrm{Y}}
\newcommand{\YT}{\mathrm{Y}\nabla}

\title{Intrinsically knotted graphs with 21 edges}
\author{Jamison Barsotti}
\author{Thomas W.\ Mattman}
\address{Department of Mathematics and Statistics,
California State University, Chico,
Chico, CA 95929-0525}
\email{JBarsotti@mail.CSUChico.edu}
\email{TMattman@CSUChico.edu}

\begin{document}

\begin{abstract}
We show that the 14 graphs obtained by $\TY$ moves on $K_7$ constitute a complete list of the minor minimal intrinsically knotted graphs on 21 edges. We also present evidence in support of a conjecture 
that the 20 graph Heawood family, obtained by a combination of $\TY$ and $\YT$ moves on $K_7$, is
the list of graphs of size 21 that are minor minimal with respect to the property not $2$--apex.
\end{abstract}

\maketitle

\section{Introduction}

We say that a graph is {\bf intrinsically knotted or IK} if every tame embedding of the graph in $\R^3$ contains a non-trivially knotted cycle. A graph is
{\bf minor minimal IK or MMIK} if it is IK, but no proper minor has this property. Robertson and Seymour's Graph Minor Theorem~\cite{RS} shows that there is a finite list of MMIK graphs. However, as it remains difficult to determine this list, research has focused on classification with respect to certain families of graphs. For example, it follows from Conway and Gordon's seminal paper~\cite{CG} that $K_7$ is the only MMIK graph on seven or fewer vertices; two groups~\cite{CMOPRW} and \cite{BBFFHL} independently determined the MMIK graphs on eight vertices; and a classification of nine vertex graphs, based on a computer search, has been announced (see \cite{Mo} and \cite{GMN}). In terms of edges, it is 
known (\cite{JKM} and, independently, \cite{Ma}) that a graph of size 20 or less is not IK. The current paper presents a classification for graphs of 21 edges.  

Kohara and Suzuki~\cite{KS} showed that the 14 graphs obtained from $K_7$ by a (possibly empty) sequence of $\TY$ moves are MMIK. We will refer to this family as the
{\bf KS graphs}.  Recall that a {\bf $\TY$ move} consists of deleting the edges of a $3$-cycle $abc$ of graph $G$, and adding a new degree three vertex adjacent to the vertices $a$, $b$, and $c$. The resulting graph $G'$ has the same size as $G$ and one additional vertex. Our main theorem asserts that the KS graphs are precisely the MMIK graphs of size 21.

\begin{theorem}  \label{thmain}%
The 14 KS graphs are the only MMIK graphs on 21 edges.
\end{theorem}

As Kohara and Suzuki already proved these graphs are MMIK, our contribution is to show that no other graph of size 21 is IK. (Graphs of size 20 are not IK, so a  connected 21 edge IK graph is also MMIK.)

We break the proof into cases by the order of the graph.  Let $G$ be a MMIK graph of size 21. We can assume $\delta(G)$, the {\bf minimum degree}, is at least three. Indeed, deleting a degree zero vertex or
contracting an edge of a vertex of degree one or two will result in an IK minor. Since a $(15,21)$ graph must have at least one vertex of degree two or less, we can assume $|V(G)| \leq 14$. Our argument is an induction starting with the case of $(14,21)$ graphs and descending to $(13,21)$ and so on.

Our induction on decreasing graph order relies on an observation essentially 
due to Sachs (see~\cite{S}): the $\TY$ move preserves intrinsic knotting. 
The reverse {\bf $\YT$ move}, delete a degree three vertex and add edges to make its neighbors mutually adjacent, does not preserve IK and this is illustrated by the Heawood Family. Following \cite{HNTY}, {\bf Heawood family} will denote the set of 20 graphs obtained from $K_7$ by a sequence of zero or more $\TY$ or $\YT$ moves. The family is illustrated schematically in Figure~\ref{figHea}  (taken from \cite{GMN}) where $K_7$ is graph 1 at the top of the figure and the $(14,21)$ Heawood graph is graph 18 at the bottom. In addition to the 14 KS graphs, the Heawood family includes six additional graphs (graphs 9, 14, 16, 17, 19, 20 in the figure) that are not IK, as was shown independently in \cite{GMN} and \cite{HNTY}. Thus, for example, the $\YT$ move from graph 5 to 9 takes an IK graph to one that is not. (That $\YT$ does not preserve IK was first observed by Flapan and Naimi~\cite{FN}).

\begin{figure}[htb]
\begin{center}
\includegraphics[scale=0.75]{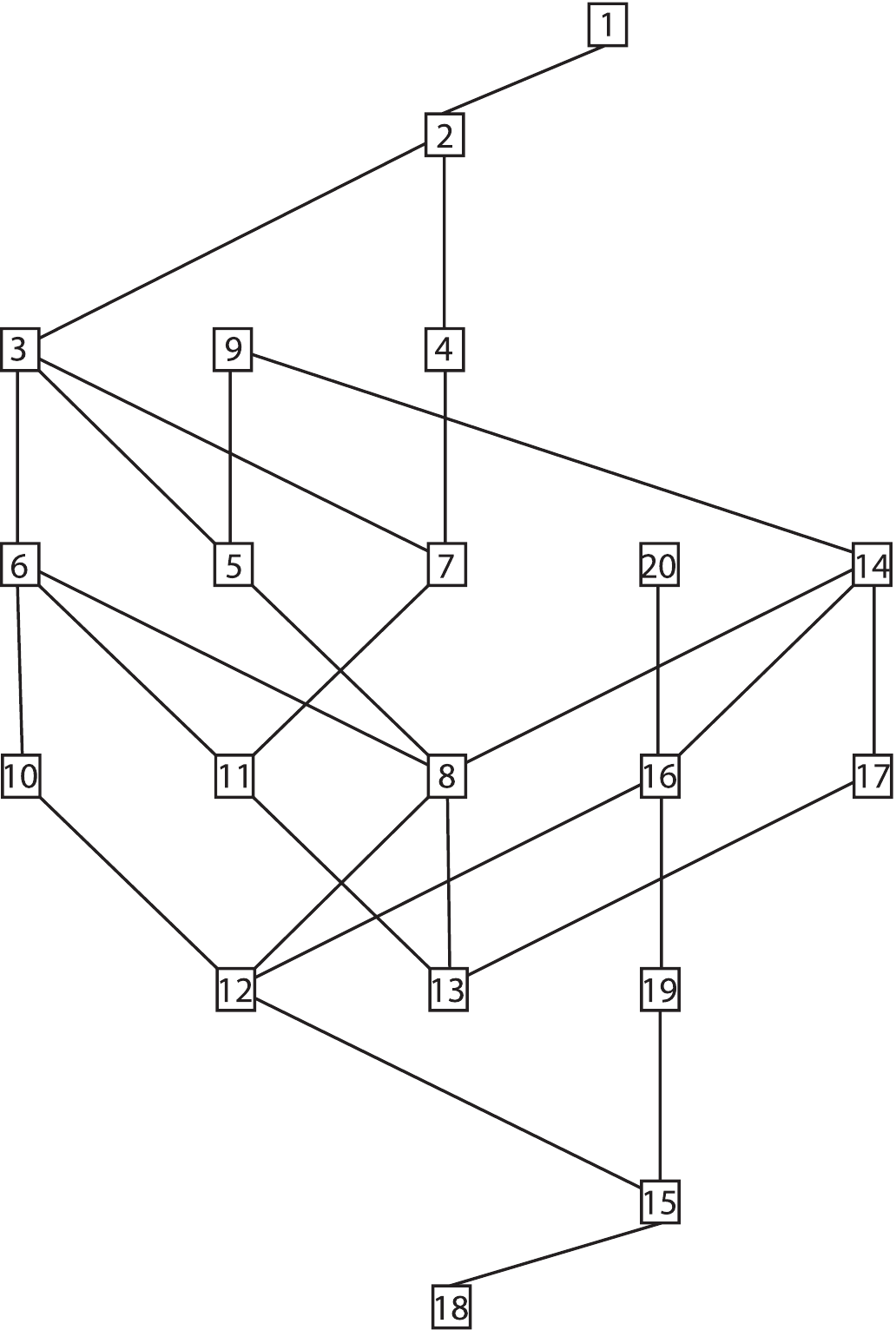}
\caption{The Heawood family (figure taken from \cite{GMN}). Edges represent $\TY$ moves.}\label{figHea}
\end{center}
\end{figure}

We conclude this introduction with some observations and questions about 21 edge graphs that are not 
$2$--apex.
Recall that a graph is {\bf $n$--apex} if it can be made planar through deletion of $n$ or fewer vertices (and their edges). Thus, a graph $G$ is {\bf not $2$--apex, or N2A} if whenever two vertices (and their edges) are deleted, the resulting graph is not planar.
We will make much use of the following lemma, which is a consequence of the observation, due independently to \cite{BBFFHL} and \cite{OT}, that the join, $H * K_2$, of $H$ and $K_2$ is IK if and only if $H$ is nonplanar. 

\begin{lemma} \label{lem2ap}%
\cite{BBFFHL, OT}
If $G$ is IK, then $G$ is N2A.
\end{lemma}

In other words the class of IK graphs is a subset of those that are N2A.
In particular, every graph in the Heawood family is N2A and it's natural to ask if there are other 21 edge 
examples. (Since size 20 graphs are $2$--apex~\cite{Ma}, a connected 21 edge N2A graph is necessarily minor minimal or MMN2A.)

\begin{ques} \label{queHMMN2A}%
 Is the Heawood family the set of graphs of size 21 that are MMN2A?
\end{ques}

The following observation allows us to answer the question for graphs of order ten or less.

\begin{prop}  \label{prop1}%
Let $G$ be a graph with either $|V(G)| \leq 8$ or else $|V(G)| \leq 10$ and $|E(G)| \leq 21$. If $G$ is  N2A and a $\YT$ move takes $G$ to $G'$, then $G'$ is also N2A.
\end{prop}

On the other hand, it is straightforward to settle the question for order 14 or more using the idea that a
MMN2A graph has minimal degree 3. Thus, all that remains are graphs of orders 11, 12, or 13.

\begin{prop} \label{prop2}%
If $G$ is MMN2A with $|E(G)| = 21$ and $|V(G)| \neq 11,12,13$, then $G$ is a Heawood graph.
\end{prop} 

In~\cite{HNTY}, the authors show that the Heawood graphs are minor minimal with respect to the property intrinsically knotted or completely $3$--linked. This suggests that property may be related to N2A.

\begin{ques} How are N2A graphs related to those that are intrinsically  knotted or completely $3$--linked?
\end{ques}

It's easy to see that $\YT$ does not preserve N2A in general. For example, the disjoint union of
three $K_{3,3}$ graphs is N2A, but applying a $\YT$ move destroys this property. However, it
may be that Proposition~\ref{prop1} can be extended to all graphs of size $21$.

\begin{ques} \label{queYTN2A}%
 Does $\YT$ preserve N2A on graphs of size 21?
\end{ques}

As $\YT$ does preserve N2A in the Heawood graphs, an affirmative answer to Question~\ref{queHMMN2A}
would imply the same for Question~\ref{queYTN2A}. If so, we could ask about the first instance of $\YT$ not preserving N2A.

\begin{ques} What is the simplest (e.g., smallest in size or order) graph $G$ that is N2A but admits
a $\YT$ move to a graph $G'$ that is $2$--apex?
\end{ques}

After introducing some preliminary lemmas in the next section, we 
devote one section each to intrinsic knotting of graphs of order $14$, $13$, $12$, $11$, and $10$, respectively.  As graphs of order nine or less were treated earlier in \cite{Ma}, taken together
this constitutes a proof of Theorem~\ref{thmain}. We conclude the paper with a proof of 
Propositions~\ref{prop1} and \ref{prop2} in Section~\ref{secN2A}.

After preparing this paper we learned that Lee, Kim, Lee, and Oh~\cite{LKLO} have also announced a proof
of Theorem~\ref{thmain}. Our approach is based on the first author's thesis~\cite{B}.

\section{Definitions and Lemmas}

As mentioned in the introduction, we prove the main theorem by induction starting with graphs of
14 vertices and working down to those having ten. We begin by observing that it is enough to consider 
triangle--free graphs.

\begin{rmk} \label{rmkTfree}
As the KS graphs are precisely the IK graphs in the Heawood family, it will be enough for us to show that size 21 MMIK graphs are Heawood. Using our induction, this allows us to reduce to the case of triangle--free graphs. Indeed, if a 21 edge MMIK graph $G$ has a triangle, apply a $\TY$ move to obtain an IK graph $G'$ with one additional vertex. This graph must be MMIK as graphs on 20 edges are not IK. Then, by the inductive hypothesis, $G'$ is Heawood, whence $G$ is also.
\end{rmk}

\begin{figure}[ht]
\begin{center}
\includegraphics[scale=0.5]{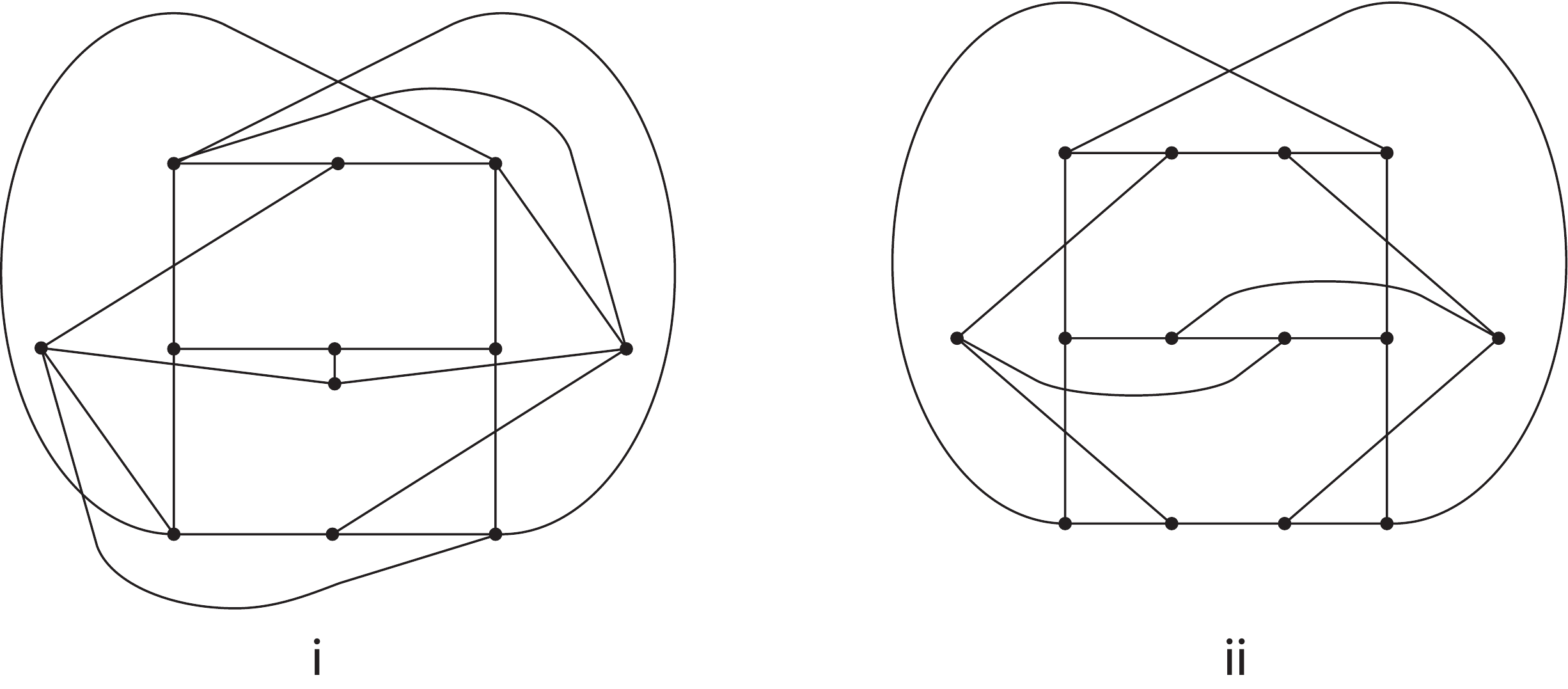}
\caption{The two triangle--free Heawood graphs. i) $H_{12}$ ii) $C_{14}$}\label{figH12C14}
\end{center}
\end{figure}

Note that only two graphs in the Heawood family are triangle--free, namely graphs 13 and 18 in Figure~\ref{figHea}. These graphs were named $H_{12}$ and $C_{14}$ by Kohara and Suzuki~\cite{KS}, see Figure~\ref{figH12C14}.

Throughout this paper, for $a,b \in V(G)$, we will use $G-a$ and $G-a,b$ to denote the 
induced subgraphs on $V(G) \setminus \{a\}$ and $V(G) \setminus \{a,b \}$, respectively.
We will also write $G+a$ to denote a graph with vertices $V(G) \cup \{a\}$ that includes $G$ as the induced subgraph on $V(G)$. In case $V(G)$ and $\{a\}$ are included in the vertex set of some larger graph, $G+a$ will mean the induced subgraph on $V(G) \cup \{a\}$.

Here is an example of how Lemma~\ref{lem2ap} and the triangle--free condition work in concert.
We'll use $|G|$ to denote the order, or number of vertices of a graph.

\begin{lemma} \label{lem2comp}%
Let $G$ be a graph with minimum degree $\delta(G) \geq 3$.  Suppose 
$\exists a,b \in V(G)$ such that $G-a,b$ has a tree component $T$. If $|T| \leq 3$ or $T$ has a degree two vertex adjacent to a leaf, then $G$ has a triangle.
\end{lemma}

\begin{proof} 
Since $\delta(G) \geq 3$, leaf vertices of $T$ are adjacent to both $a$ and $b$ in $G$ while degree two vertices are adjacent to at least one. Thus, 
under the hypotheses on $T$, a triangle is formed, with a leaf of the tree and one of $a$ and $b$ constituting two of the triangle's vertices.
\end{proof}

In other words, if $G$ is MMIK, triangle-free, and of order 21, then, by Lemma~\ref{lem2ap}, $G$ is N2A. So $\forall a,b \in V(G)$, $\Gab$ is nonplanar. Lemma~\ref{lem2comp} restricts the structure of any tree components of these nonplanar graphs. Our strategy is to combine enough restrictions of this type to either force a contradiction or else demonstrate that $G$ is $H_{12}$ or $C_{14}$.

As above, we will often encounter non-planar graphs of the form $\Gab$. Although this means $\Gab$ has either a $K_5$ or $K_{3,3}$ minor by Kuratowski's theorem, the $K_{3,3}$ case is more important in our 
argument. Especially, we will often encounter {\bf split $K_{3,3}$'s}, graphs obtained from
$K_{3,3}$ by a finite (possibly empty) sequence of vertex splits.  Conversely, this means that starting from a split $K_{3,3}$  graph $G$, we can recover a $K_{3,3}$ minor by repeatedly deleting vertices using the following two {\bf deletion operations} until there remain no vertices of degree less than three.
\begin{description}
\item[D1] Delete a vertex of degree one and its edge.
\item[D2] Delete a degree two vertex $b$ replacing its edges $ab$ and $bc$ with a new edge $ac$. 
\end{description} 
We will refer to the six vertices in $G$ that survive this sequence
of deletions as the {\bf orginal vertices.} 
Since there may be more than one sequence of deletion moves leading 
to $K_{3,3}$, in general, there's more than one way to choose original vertices.
As our argument does not depend on this choice, we'll often assume, without further 
explanation, that a specific choice has been made.
An {\bf original $4$--cycle} is a cycle $C$ in $G$ that passes through exactly four original vertices. 
The {\bf split $4$--cycle} of $C$ is the
component of $C$ in $G - v,w$ where $v$ and $w$ are the 
two original vertices not in $C$.

In addition to D1 and D2 we will have occasion to refer to D0, meaning deletion of an isolated vertex. 
The {\bf simplification} of a graph $G$ is the graph $G'$ with $\delta(G') \geq 3$ formed by repeatedly applying the three deletion operations to $G$. Although, in principle, $G'$ may be a multi-graph (e.g., applying D2 to a vertex in a three cycle will lead to a double edge), that won't happen in the examples we consider in this paper. For example, as mentioned above, the simplification of a split $K_{3,3}$ is $K_{3,3}$.

\begin{figure}[ht]
\begin{center}
\includegraphics[scale=0.35]{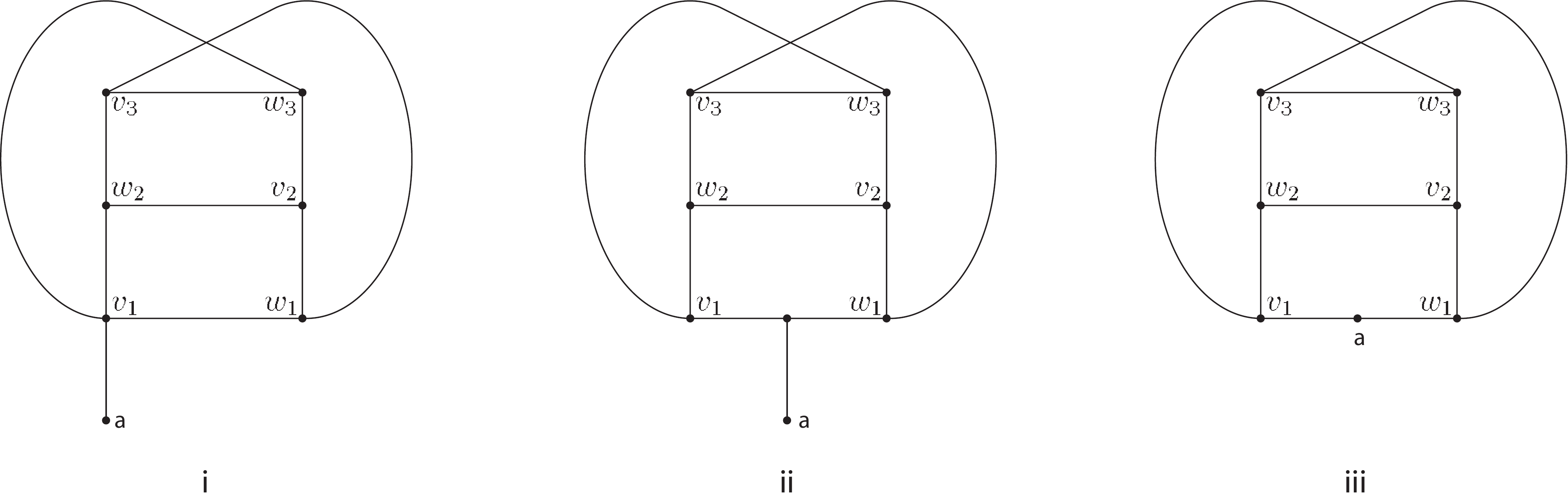}
\caption{The simplification of a split $K_{3,3}$ relative to $a$.}\label{figK33nrprt}
\end{center}
\end{figure}

We'll also use the idea of simplification relative to a vertex.
Let $G$ be a split $K_{3,3}$ and $a \in V(G)$. The {\bf simplification of $G$ relative to $a$}, $G|_{a}$, is the graph formed 
by repeatedly applying D1 and D2 to delete vertices other than $a$ until all 
vertices (except possibly $a$) have degree at least three.
Then either $a$ is an original vertex or else  $G|_{a}$ is one of the three graphs of Figure~\ref{figK33nrprt}. In case $a=v_1$ is an original vertex of $G$ or we
have the graph of Figure~~\ref{figK33nrprt}i, we say that $v_1$ is the {\bf nearest part} of 
$K_{3,3}$ to $a$. In the case of Figure~\ref{figK33nrprt}ii or iii, we will say that the edge $v_1w_1$ is the {\bf nearest part} of $K_{3,3}$ to $a$.  

Here's an alternative characterization of split $K_{3,3}$ graphs. We use $\chi(G) = |G| - \|G\|$ to denote the {\bf Euler characteristic of a graph}, that is, the difference between the $|G| = |V(G)|$ and $\|G\|=|E(G)|$.

\begin{lemma}
\label{lemsplitk33} 
A graph $G$ is a split $K_{3,3}$, if and only if, it is connected with a $K_{3,3}$ minor and $\chi(G)=-3$.
\end{lemma}

\begin{proof}
Assume $G$ is a split $K_{3,3}$. Since $G$ can be made using a series of vertex splits on a $K_{3,3}$ graph, then it is connected and has a $K_{3,3}$ minor. Since each vertex split adds exactly one vertex and one edge, $\chi(G)= \chi(K_{3,3})=-3$.

Now assume $G$ has a $K_{3,3}$ minor, is connected, and that $\chi(G)=-3$. So $G$ can be built by adding vertices and edges to a 
split $K_{3,3}$, $H$. Since $\chi(G) = -3 = \chi(H)$ an equal number of
vertices and edges are added. And as $G$ and $H$ are both connected, we can build $G$ from $H$ through a sequence of connected graphs as follows.
At each step we add a vertex along with one of its edges so as to connect
the new vertex to the connected graph of the previous step.
In other words, $G$ is obtained from $H$ by a series of vertex splits. Thus, $G$ is also a split $K_{3,3}$.
\end{proof}

Starting from $G$ MMIK, Lemma~\ref{lem2ap} implies that every $G-a,b$ is non-planar. The next two lemmas show that, if $G-a,b$ is a split $K_{3,3}$ then $a$ and $b$ both must have
independent paths to each of the original vertices.

\begin{lemma}
\label{lemapaths}%
Let $G$ be a split $K_{3,3}$. The graph $G+a$ is $1$--apex if there is an original vertex, $v$, such that every path from $a$ to $v$ contains another original vertex.
\end{lemma}

\begin{proof} 
Consider $G+a$ where $G$ is a split $K_{3,3}$, and say that $v \in V(G)$ is an original vertex such that any path from $a$ to $v$ contains another original vertex. Let  
$w$ be an original vertex that is adjacent to $v$ in the underlying $K_{3,3}$. If $a$ is adjacent to $b \in V(G-{v,w})$, then either $b$ is on the split $4$-cycle in $G - {v,w}$ or else $b$ is a vertex
that has $w$ as its nearest part in the underlying $K_{3,3}$.
It follows that $(G + a) - w$ is planar and $G+a$ is $1$-apex.
\end{proof}


\begin{lemma} 
\label{lemapaths2}%
Suppose  $G$ is N2A and  $G^{\ast}  = G-a,b$ is a split $K_{3,3}$ for some $a,b \in V(G)$. Then, in $G^{\ast} + a$, for each original vertex $v$, there is a path from $a$ to $v$ that avoids the other original vertices,  and similarly for $G^{\ast} + b$.
\end{lemma}

\begin{proof} Since $G$ is N2A, $G^{\ast} + a$ is not $1$-apex. Apply Lemma~\ref{lemapaths}.
\end{proof}

We conclude the introduction by characterizing graphs formed by adding a vertex of degree three or four to a split $K_{3,3}$. 

\begin{figure}[ht]
\begin{center}
\includegraphics[scale=0.5]{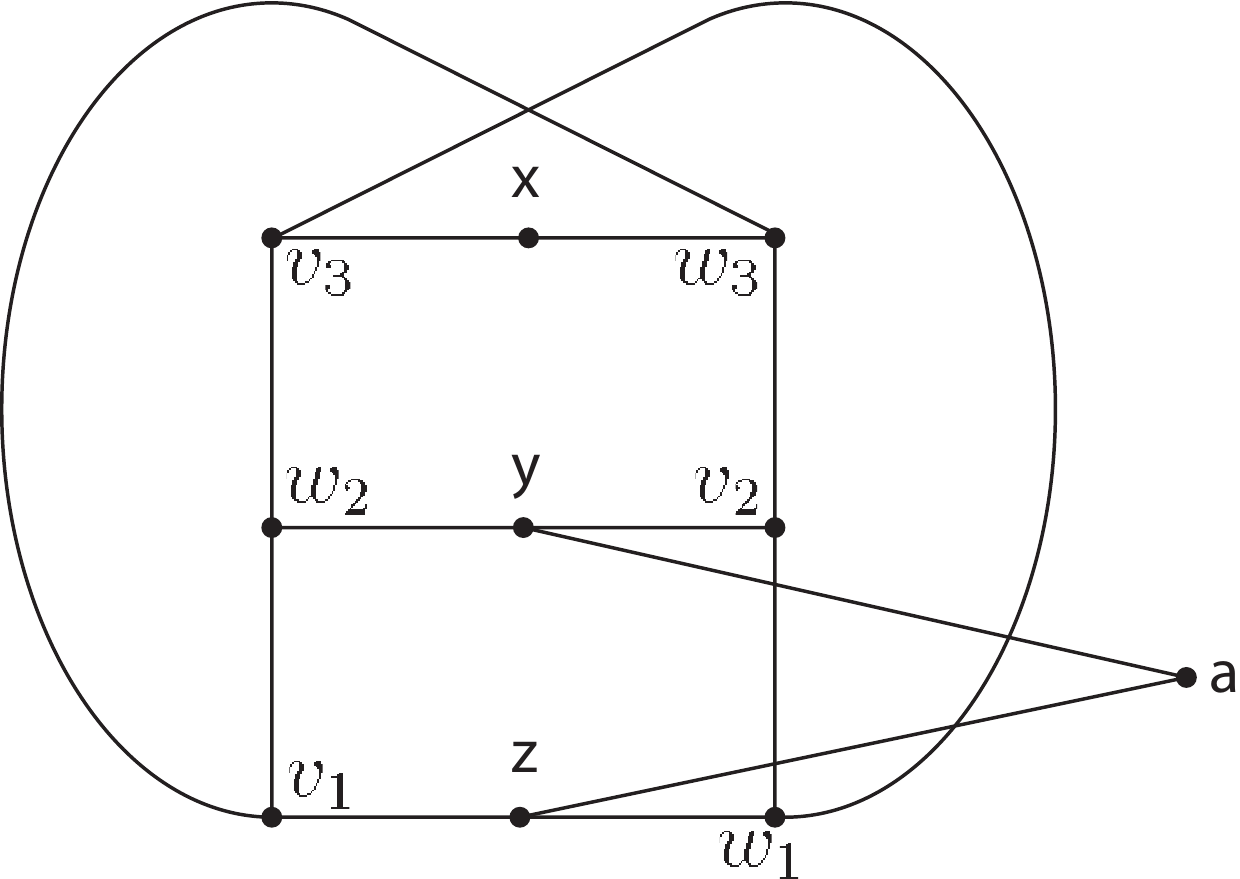}
\caption{Adding a degree $3$ vertex to a split $K_{3,3}$.}\label{figda3}
\end{center}
\end{figure}

\begin{lemma}
\label{lemda3}%
If $G+a$ is formed by adding a degree three vertex $a$ to a split $K_{3,3}$ graph $G$ and $G+a$ is not $1$-apex, then the simplification of $G+a$ is the graph of Figure~\ref{figda3}.
\end{lemma}

\begin{proof}
By Lemma~\ref{lemapaths}, there are paths from $a$ to each original vertex that avoid all other original vertices. Let $N(a) = \{n_1,n_2,n_3\}$.
As there are six vertices and $d(a) = 3$, then each 
$n_i$ must have an edge as its nearest part, and up to relabeling of  the original vertices, $n_i$ has the edge $v_iw_i$ of $G$ as its nearest part. 
This means the simplification of $G+a$ is the graph of Figure~\ref{figda3}.
\end{proof}

\begin{figure}[ht]
\begin{center}

\includegraphics[scale=.5]{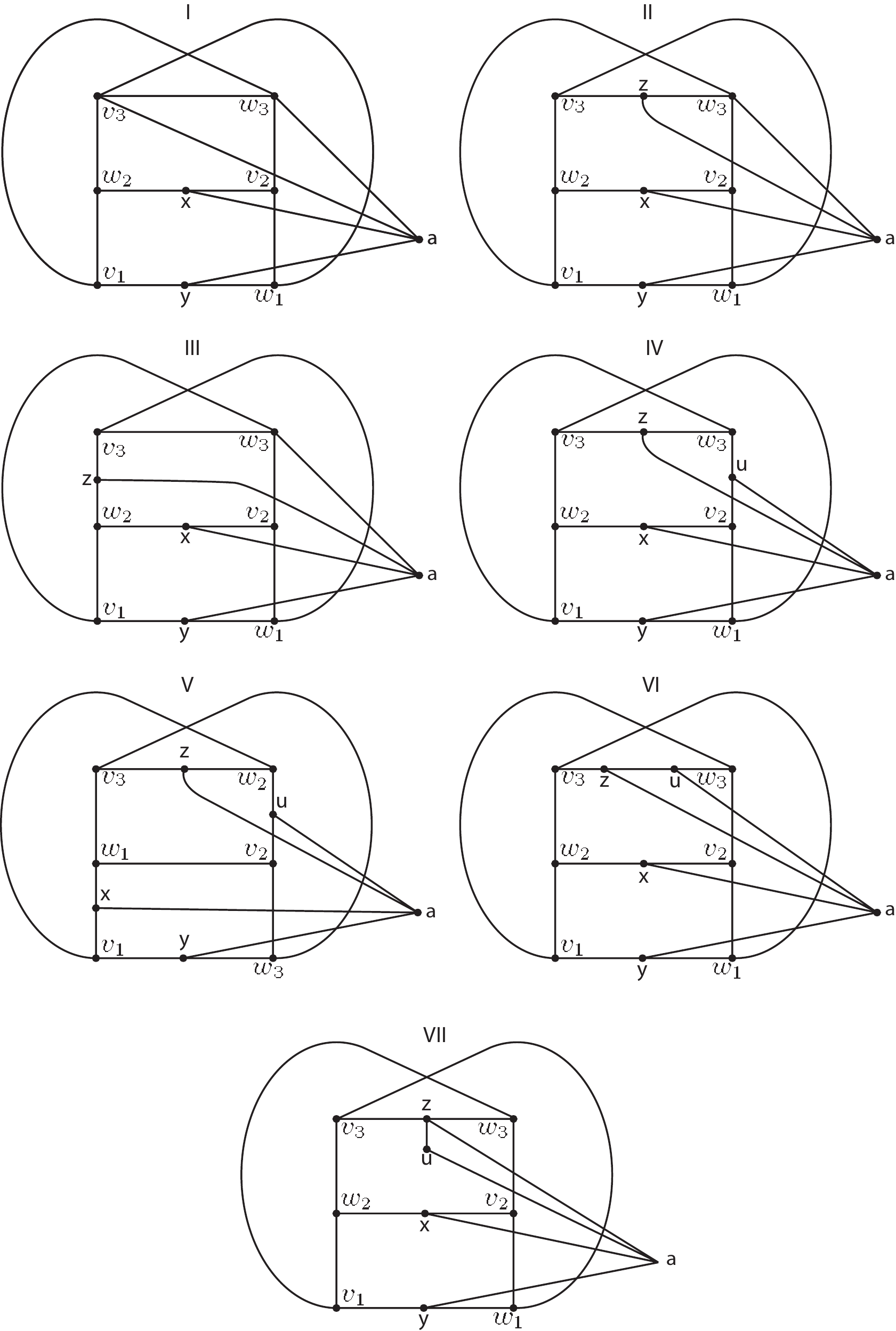}
\caption{Adding a degree $4$ vertex to a split $K_{3,3}$.}\label{figda4}
\end{center}
\end{figure}


\begin{lemma}
\label{lemda4}%
If $G+a$ is formed by adding a vertex $a$ of degree four to a split $K_{3,3}$ graph $G$ and $G+a$ is not $1$-apex, then $G+a$ is one of the seven graphs in Figure~\ref{figda4}.
\end{lemma}

\begin{proof}
By Lemma~\ref{lemapaths}, there are paths from $a$ to each original vertex that avoid all other original vertices. 
Let $N(a) = \{n_1,n_2,n_3,n_4\}$.
As there are six vertices and $d(a) = 4$, then there is an 
$n_i$, say $n_1$, that has an edge, say $v_1w_1$, as its nearest part. Since there are four original vertices left and three neighbors of $a$, another $n_i$, say $n_2$, must have an edge
as its nearest part with vertices disjoint from $\{v_1, w_1\}$, call it $v_2 w_2$. 
There are three graphs generated
when $a$ has a neighbor whose nearest part is an original vertex of $G$ and four more when $a$ has no such neighbor. Figure~\ref{figda4} shows the graphs that results from this condition.
\end{proof}

\section{14 vertex graphs}

In this section, we will show that the only 14 vertex MMIK graph on 21 vertices is the KS graph $C_{14}$ (Figure~\ref{figH12C14}ii).
See \cite{KS} for the names, such as $C_{14}$, of the KS graphs. We first characterize N2A graphs.

\begin{prop} 
\label{prop14N2A}%
Let $G$ be a connected $(14,21)$ graph. If $G$ is N2A, then $G$ is the KS graph $C_{14}$.
\end{prop}

\begin{proof} Let $G$ be a connected $(14, 21)$ graph and assume $G$ is N2A.
If a $21$ edge graph $G$ has $\delta(G) < 3$ then, by applying a deletion operation, $G$ simplifies to a graph with fewer than $21$ edges and is therefore $2$-apex. Since $14 \times 3 =  2 \times 21$,  $G$ must have the degree sequence $(3^{14})$. For any vertex $a$, $G-a$ has degree sequence $(3^{10},2^3)$. Now choose another vertex, $b$, such that $G^{\ast} = G-a,b$ has the sequence $(3^6,2^6)$ (i.e., $a$ and $b$ have no common neighbors). There are enough degree $3$ vertices in $G-a$ to assure we can always choose such a $b$.

Since  $G$ is N2A and $G^*$ has the sequence $(3^6,2^6)$, then $G^*$ must be a split $K_{3,3}$.
By Lemma~\ref{lemda3}, $G^*+a$
simplifies to Figure~\ref{figda3}. Then $G' = (G^{\ast}+a) - w_3$ is another split $K_{3,3}$.

\begin{figure}[ht]
\begin{center}
\includegraphics[scale=0.5]{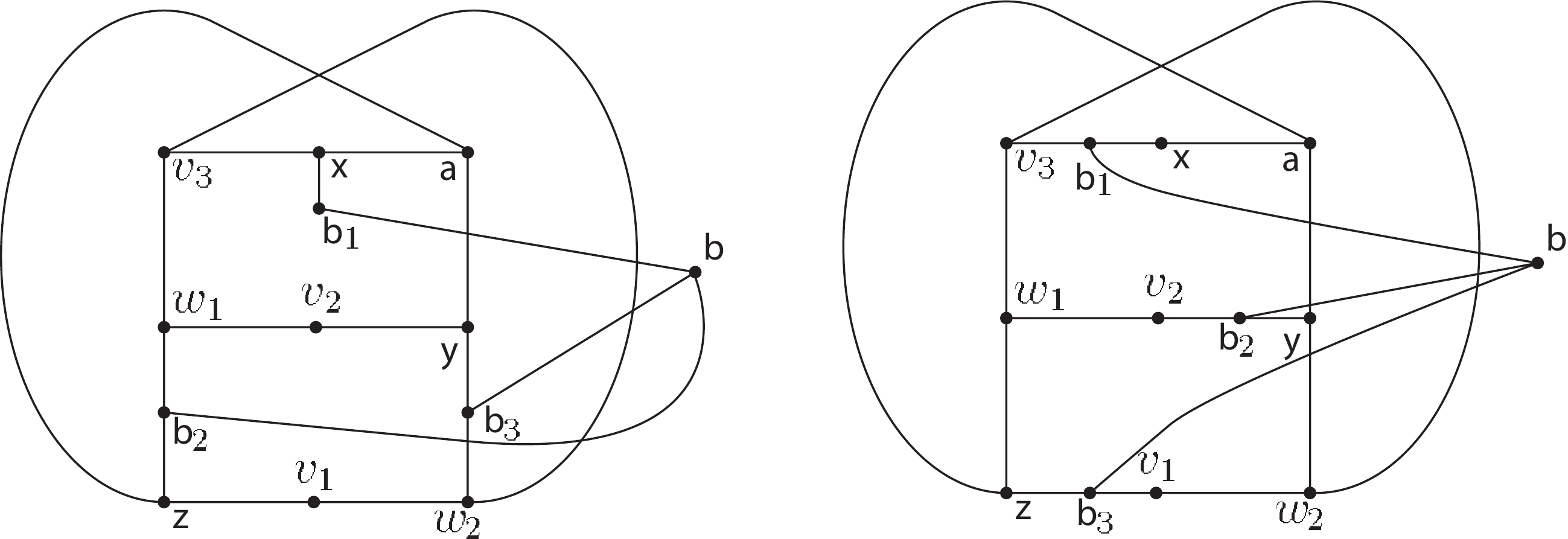}
\caption{Two possibilities for $G^\prime+b$.}\label{figda3andb}
\end{center}
\end{figure}

By  Lemma~\ref{lemda3}, $b$ must have a path to $a$ that avoids $v_3$, $w_1$, $w_2$, $y$ and $z$.
Since $a$ and $b$ have no common neighbors, this means $b$ has a neighbor $b_1$ that is adjacent
to $x$. So, there are two cases: in $G'+b$, either $b_1$ is of degree two, or else it has $v_3$ as a third neighbor. (See Figure~\ref{figda3andb}.)

In either case, $b_1$ gives paths from $b$ to the original vertices $a$ and $v_3$ and there are three ways to split the remaining four original vertices into two pairs. However, we see that $G- w_2,z$ is planar 
(and $G$ is $2$--apex), unless we make the choices shown in Figure~\ref{figda3andb}.
In both cases, adding $w_3$ back will give us $C_{14}$.
Hence the only connected (14,21) graph that is N2A is $C_{14}$.
\end{proof}

\begin{cor}
The only MMIK $(14,21)$ graph is the KS graph  $C_{14}$.
\end{cor}

\section{13 vertex graphs}

\begin{prop} The only MMIK $(13,21)$ graph is the KS graph $C_{13}$.
\end{prop}

\begin{proof}
Let $G$ be an MMIK $(13,21)$ graph. 
As in Remark~\ref{rmkTfree}, if $G$ has a triangle, it must be $C_{13}$,
which is MMIK. So, we will assume
$G$ is triangle-free and force a contradiction (generally, by arguing $G$ must have a triangle).

An MMIK graph $G$ will have $\delta(G) \geq 3$ and one of the following three degree sequences: $(3^{12},6)$, $(3^{11},4,5)$, or $(3^{10},4^3)$.

Case 1: $(3^{12},6)$ 

Assume $G$ has degree sequence $(3^{12},6)$. Delete $a$ and $b$ not adjacent with $d(a) = 6$, $d(b) = 3$. Then $\| G - a,b \|= 12$ and by \cite{Ma} if $G-a,b$ is not planar, it
has  a $K_2$ component, which, as in Lemma~\ref{lem2comp} results in a triangle in $G$, a contradiction.

Case 2: $(3^{11},4,5)$

Assume $G$ has degree sequence $(3^{11},4,5)$. Delete the degree five and four vertices $a$ and $b$. If $a$ and $b$ are not adjacent, then, as in the previous case, $G$ has a triangle. So, 
we can assume $a$ and $b$ are adjacent. Then $\| G -a,b \| = 13$ and by \cite{Ma} if $G-a,b$ is not planar, it is either $K_5 \cup K_2 \cup K_2 \cup K_2$, in which case $G$ has a triangle,
or has a component with $K_{3,3}$ minor as well as at least one tree component. However, a leaf of a tree component will form a triangle with $a$ and $b$. In either case, we deduce that $G$ has a triangle, a contradiction.

Case 3: $(3^{10},4^3)$

Assume $G$ has degree sequence $(3^{10},4^{3})$. Delete two degree four vertices $a$ and $b$. Assume $a$ and $b$ are not adjacent;
then, $\| G -a,b \| = 13$ and by \cite{Ma} if $G-a,b$ is not planar, it is either $K_5 \cup K_2 \cup K_2 \cup K_2$, in which case $G$ has a triangle,
or has a component with $K_{3,3}$ minor as well as at least one tree component, $T$. If $|T|\leq 3$ then, by Lemma~\ref{lem2comp}, $G$ has a triangle, which is a contradiction. So we'll assume $|T|>3$ and we have two cases: $|T|=4$ or $|T|=5$. (There are at least six vertices in the $K_{3,3}$ component, so at most five
of the 11 vertices in $G-a,b$ left over for $T$.)

Since $\chi(G-a,b) = -2$, there are exactly two components, $T$ and the component with $K_{3,3}$ minor, call it $H$. Moreover, $H$ is
connected with $\chi(H) = -3$ and, therefore, a split $K_{3,3}$ by Lemma~\ref{lemsplitk33}. If $T$ is not a star, Lemma~\ref{lem2comp} 
shows that there is a triangle, which is a contradiction. So we'll
assume $T$ is a star.

If $|T|=4$ then $a$ must be adjacent to all three of its leaves. 
Let $v$ be an original vertex of $H$.
The fourth neighbor of $a$ is either the fourth vertex of $T$ or 
a vertex in $H$. In either case, $G - b,v$ is planar and $G$ is $2$--apex, 
hence not IK.

If $|T|=5$ both $a$ and $b$ are adjacent to the four leaf vertices of $T$ 
and have no neighbors in $H$. So $G$ is not connected and thus not MMIK,
again, a contradiction.

Say that a graph $G$ has this sequence but there does not exist a pair of degree four vertices, $a$ and $b$, such that $a$ and $b$ are not adjacent. Then the three degree four vertices form a triangle in $G$, which is a contradiction. This completes the argument for Case 3 and with it the proof of the proposition.
\end{proof}

\section{12 vertex graphs}
\begin{prop} The only MMIK $(12,21)$ graphs  are the KS graphs $C_{12}$ and $H_{12}$.
\end{prop}

\begin{proof}
Suppose $G$ is a MMIK $(12,21)$ graph. As in Remark~\ref{rmkTfree}, if $G$ has a triangle, it must
be $C_{12}$. So we'll assume that $G$ is triangle-free and either show $G$ is $H_{12}$ or else 
deduce a contradiction. Note that $H_{12}$ has degree sequence $(3^6,4^6)$

Let us first consider a $(12,21)$ MMIK graph $G$ such that there exists a pair of vertices $a$ and $b$ with $\|G-a,b\| < 13$. By an Euler characteristic argument, if
$G-{a,b}$ is nonplanar, then it contains at least one tree, $T$, and $|T| \leq 4$. By Lemma~\ref{lem2comp}, 
unless $T$ is a star on four vertices, this means $G$ has a triangle, which is a contradiction.
So we will assume that in the graph $G-{a,b}$, the tree component $T$ has order four and is a star. This implies that the other component must be the graph $K_{3,3}$. Adding the vertex $a$ back into the graph, we see that $a$ needs to be adjacent to all the leaves of $T$. Also, by Lemma~\ref{lemapaths2}, $a$ must be adjacent to every vertex in the $K_{3,3}$. Hence $d(a) \geq 9$ and $\|G-{b}\| \geq 21$ which is impossible. So, if $G-a,b$ has size less than 13, we have a contradiction.

This helps us narrow down the degree sequences we have to consider. For instance, suppose $G$ is a $(12,21)$ graph and has a vertex $a$, such that, $d(a)>5$. Then there is another vertex $b$, 
such that, $b$ is not adjacent to $a$ and $d(b)>2$, so $\|G-a,b\| < 13$, leading to a contradiction, as above. Similarly, if there is a vertex of $a$ of degree five and 
another vertex $b$ such that either $d(b)=5$, or else $d(b)=4$ and $b$ is not adjacent to $a$, then again $\|G - a,b \| < 13$ and we have a contradiction. 
Recall that $G$ MMIK implies $\delta(G) \geq 3$.
In order to avoid a triangle among vertices of degree four or more, it remains only to consider the two cases where $G$ has the degree sequence $(3^{7},4^4,5)$ or $(3^{6},4^{6})$.

Case $1$: $(3^{7},4^4,5)$

Assume $G$ has degree sequence $(3^{7},4^4,5)$. Denote the vertex of degree five as $a$ and recall that it must be adjacent to all the vertices of degree four. 
Delete $a$ and note that $G-{a}$ has the sequence $(2,3^{10})$. Next, delete $b$ such that the degree of $b$ in $G$ was four.
Notice that if $b$ is adjacent to the degree two vertex in $G-{a}$, that would imply a triangle in $G$, so we assume it is not. Then $G-{a,b}$ has the degree sequence $(2^4,3^6)$. If $G-{a,b}$ is nonplanar then, since $\chi(G-{a,b}) = -3$, it contains a split $K_{3,3}$ and if it is not connected, its other component is a cycle of order $3$ or $4$.  

As $a$ has degree five in $G$ and $b$ is one of its neighbors, $a$ has four neighbors in $\Gab$, exactly one of them being a vertex of degree two. This means, in contradiction to Lemma~\ref{lemapaths2}, there is at least one original vertex of the split $K_{3,3}$ that has no path to $a$ that avoids the other original vertices. 
The contradiction shows there is no such graph with degree sequence $(3^{7},4^4,5)$

Case $2$: $(3^{6},4^{6})$

This case is hard as $H_{12}$ has this degree sequence
and $H_{12}$ is a triangle-free MMIK graph.
We can continue to eliminate 
many cases that result in a triangle, but in the end we will
need to explicitly show that a MMIK graph with this sequence is $H_{12}$.
We will delete vertices $a$ and $b$ both of degree four, which we can assume to be nonadjacent. Let $G^* = G - {a,b}$.  Notice that $\chi(G^*)=-3$ and we have two cases:
either $G^*$ is connected or it is not.

Assume that $G^*$ is nonplanar and is not connected. Using Lemma~\ref{lem2comp},
$G^*$ is either a nonplanar $(7,10)$ graph together with a cycle of order three, a nonplanar $(6,9)$ graph together with a cycle of order four, or a nonplanar $(6,10)$ graph together with a star of order four.  In the first case, a cycle of order $3$ is a triangle in $G$, a contradiction. 
In the case where $G^*$ has a cycle of order $4$, denote it by $C$, then the other component is $K_{3,3}$. Since one of $a$ and $b$ has at least two neighbors on $C$, adding that vertex and deleting any vertex of the $K_{3,3}$ in $G^*$ results in a planar graph meaning $G$ is $2$--apex, contradicting $G$ MMIK.
Finally, if $G^*$ has a star of order $4$, then both $a$ and $b$ are adjacent to each of the three leaves of the star. The nonplanar $(6,10)$ has minimal degree one or more and must be a $K_{3,3}$ with an extra edge, call it $v_1v_2$. Then $G- a,v_1$ is planar and $G$ is $2$--apex, contradiction.
Thus, we conclude that if $G^*$ is not connected, then $G$ is not MMIK.

We will now assume that $G^*$ is connected. So by Lemma~\ref{lemsplitk33}, $G^*$ is a split $K_{3,3}$. Then by Lemma~\ref{lemda4}, we see that $G^*+a$  and $G^*+b$ simplify
to one of the seven graphs in Figure~\ref{figda4}. We shall denote our graphs as in the figure and use the vertex labels given there for convenience.

Notice that in the cases of VI and VII, $G^*+a$ is the graph shown (no additional vertex splits are needed) and $G$ has a triangle, a contradiction.
So we will assume that $G^*+a$ (and, similarly, $G^*+b$) simplify to one of the other five graphs. If $G^*+a$ simplifies to V, then $|G^*+a|=11$, so we do not  have any additional vertex splits
(and $G^* + a$ is as shown). Deleting the vertices labeled $z$ and $x$ in the figure, the resulting graph has a planar representation. Furthermore, if $b$ is not a neighbor of $a$, $b$ can be a neighbor to all the other remaining vertices and maintain the graph's planarity. Since our assumption was that $a$ and $b$ are not neighbors, we have shown that G is $2$-apex in the case where $G^*+a$, or, by symmetry, $G^*+b$, simplifies to graph V in Figure~\ref{figda4}.

Going on to the next possibility, assume $G^*+a$ simplifies to IV. Since $|G^*+a|=11$ we do not have any vertex splits. If $b$ is not a neighbor of $y$, then $G-v_1,w_1$ is planar. So assume that $y$ and $b$ are adjacent in $G$. If $b$ is not adjacent to $x$ or if $b$ is not adjacent to $z$ then $G-{y,z}$ and $G-{y,x}$ are planar respectively. Thus $b$ will have $x$, $y$, and $z$ as neighbors. If its fourth neighbor is not $u$, then $G$ will have a triangle. This shows that both $a$ and $b$ will have $x$, $y$, $z$, and $u$ as neighbors. 
But then $G-y,x$ is planar. So, if $G^*+a$ or $G^*+b$ simplifies to IV in 
Figure~\ref{figda4}, then $G$ is not MMIK

Considering the case where $G^*+a$ simplifies to to III in Figure~\ref{figda4}, we notice that III has ten vertices and $G^*+a$ has eleven vertices. This implies that $G^*+a$ is III with a 
vertex split. We will denote the vertex created by this split $u$ and refer to $u$ as the {\bf vertex split}. (In other words, much
as the deletion moves D1 and D2 allow us to imagine edge contractions as vertex deletions, we tend to think of a vertex split in terms of adding a vertex.) Notice that deleting $w_3$ and $z$ from $G^*+a$ gives us a planar graph, unless both $a$ and $b$ have $u$ as a neighbor. Assume $u$ is a neighbor of 
both $a$ and $b$ and recall that $G^*+b$ simplifies to one of  graph I, II, or III. We can rule out II, since that 
would require another vertex split. We then see that in the graph $G^*+b$, the neighborhood of $b$, after deleteing $u$ (by deletion move D2), is $\{x,y,z,w_3\}$, $\{x,y,v_3,w_3\}$, or $\{x,z,v_2,w_3\}$. In all of these cases, if we choose to delete $x$ and $w_3$ we will get a planar graph even if we add $a$ and $b$ back in, since they both have $u$ as a neighbor. Hence, in the case where $G^*+a$ or $G^*+b$ 
simplifies to III in Figure~\ref{figda4}, $G$ is $2$-apex.

Next suppose $G^*+a$ simplfies to graph II in Figure~\ref{figda4}. Notice, as when we considered graph III, there is a vertex split, $u$, on $G^*a$ not shown 
in II. In II, we see that the vertices $z$, $a$, and $w_3$ form a triangle, so we need only consider the graphs for which $u$ is on one of the edges of this triangle. Assume $u$ is 
between $z$ and $b$. We see that $u$ is a neighbor of both $a$ and $b$. If $G^*+b$ is topologically equivalent to graph II, either $G$ contains a triangle (a contradiction), or else the neighborhood of $b$ is 
$\{u,x,y,w_3\}$ or $\{u,x,y,v_3\}$. For both choices of $b$'s neighborhood, $G-{w_3,v_3}$ is planar. So we assume that $G^*+b$ simplifies to graph I in Figure~\ref{figda4}. Then, $b$ is adjacent to $u$ and $u$ is adjacent to $z$, so $b$ is adjacent to $x$ or $y$. Without losing generality, we can say 
that $b$ is adjacent to $x$. Hence, $b$ also has $w_1$ and $v_1$ as neighbors. Clearly, $G-w_1,v_1$ is planar.

We shall now assume that $u$ is between $z$ and $w_3$. Again, $u$ is adjacent to $b$. Not considering cases that would give us triangles, $b$ has the 
neighborhood $\{u,y,x,v_3\}$ if $G^*+b$ simplifies to II in Figure~\ref{figda4}, or else $b$ has the neighborhood $\{u,y,w_2,v_2\}$ or $\{u,x,w_1,v_1\}$ if $G^*+b$ simplifies to I. The graphs $G-{w_3,v_3}$, $G-{w_2,v_2}$, and $G-{w_1,v_1}$ are planar in each of these respective cases.

Next, suppose $u$ is between $w_3$ and $a$. Notice that $b$ is adjacent to $u$, so whether $G^*+b$ 
simplifies to graph I or II in Figure~\ref{figda4}, $b$ will 
have $x$ and $y$ as neighbors. Thus $G-{w_3,v_3}$ is planar.  
Thus, when $G^*+a$ or $G^*+b$ simplifies to II in 
Figure~\ref{figda4}, we arrive at a contradiction.

Lastly, we approach the case where $G^*+a$ simplifies to graph I in Figure~\ref{figda4}. Notice again the triangle formed between $a$, $w_3$, and $v_3$, implies 
there is a vertex split, denote it by $z$, on an edge of the triangle. Obviously, the cases where $z$ is between $a$ and $v_3$ and between $a$ and $w_3$ are symmetric to one another. We next show that they are also symmetric to placing $z$ in between $v_3$ and $w_3$. Indeed, let $H_I$ denote the graph of Figure~\ref{figda4}I. Note that $H_I - a$ and $H_I - v_3$ are 
isomorphic and the identification extends to an isomorphism of $H_I$ that interchanges $a$ and $v_{3}$. This isomorphism shows that a $z$ on $v_3 w_3$ is symmetric to one on $a w_3$.
So, without loss of generality, we will assume that $z$ is between $v_3$ and $w_3$. 

We still have another vertex split, $u$, somewhere on our graph. If we delete $v_3$ and $w_3$, we notice that as long as both $a$ and $b$ are not both adjacent to $u$, then the graph is 
planar. Vertex $b$ is adjacent to $z$ because $z$ has degree $3$ in $G$ and $b$ is also adjacent to $u$, which is is a neighbor of $x$ or $y$ since $G^*+b$ 
is topologically equivalent to I. In either case $b$ is also adjacent to $v_1$ and $w_1$ or $v_2$ and $w_2$ respectively and the graph formed is $H_{12}$ (see Figure~\ref{figH12C14}).

Therefore, if $G$ is a MMIK $(12,21)$ and has no triangle, it will be $H_{12}$. 
\end{proof}

\section{ 11 vertex graphs}

\begin {prop} The only MMIK $(11,21)$ graphs are the KS graphs, $H_{11}$, $E_{11}$, and $C_{11}$.
\end{prop}

\begin{proof}

As in the previous sections, we will use that if an $(11,21)$ graph $G$ is MMIK, 
then it has a minimum degree of at least three, and, following Remark~\ref{rmkTfree}, we assume $G$ is 
triangle--free and look for a contradiction.

By Lemma~\ref{lem2ap}, $G$ is $2$--apex and no $\Gab$ is planar. So assume that we delete
two vertices, $a$ and $b$, and in the process we also delete at least ten edges. The resulting graph $\Gab$ has order $|\Gab \ = 9$ and size $\| \Gab \| \leq 11$ and a minimum degree of 
at least one. Thus $\chi(G-{a,b})\geq-2$. Since $\chi(K_{5})=-5$ then our graph cannot have a $K_5$ minor since that would require at least three trees and we do not have enough vertices. 
(Since $\delta(G-a,b) \geq 1$, a tree has at least two vertices.) 
If $G-{a,b}$ is non-planar it must have a $K_{3,3}$ minor. Now, $\chi(K_{3,3})=-3$ so $G-{a,b}$ will have at least one tree, which must be of order two or three.
By Lemma~\ref{lem2comp} this means $G$ has a triangle, which is a contradiction.
So there can be no pair of vertices $a$ and $b$ that result in the deletion of ten or more edges.

There are six degree sequences that satisfy this condition on the deletion of two vertices: 
$(6,4^6,3^4)$, $(5^4,4,3^6)$, $(5^3,4^3,3^5)$, $(5^2,4^5,3^4)$, $(5,4^7,3^3)$, and $(4^9,3^2)$.

Case $1$: $(6,4^6,3^4)$

Assume the graph $G$ has $(6,4^6,3^4)$ as its degree sequence. Notice that if $\not\exists a,b$ such that $\| \Gab \| \leq 11$,
then the vertex of degree six is a neighbor of each vertex of degree four. Since there must be a pair of adjacent degree four vertices in $G$, then $G$ has a triangle, a contradiction.

Case $2$: $(5^4,4,3^6)$ and $(5^3,4^3,3^5)$

Assume the graph $G$ has either the degree sequence $(5^4,4,3^6)$ or $(5^3,4^3,3^5)$. It's apparent that if we cannot delete an $a$ and $b$ from $G$ such that 
$G-{a,b}$ has $11$ edges, then all the vertices of degree five are mutually adjacent. Hence there is a triangle in $G$, a contradiction.

Case $3$: $(5^2,4^5,3^4)$ and $(5,4^7,3^3)$

Assume that $G$ has either $(5^2,4^5,3^4)$ or $(5,4^7,3^3)$ as its degree sequence. We choose to delete two vertices $a$ and $b$ such that the degree of $b$ is $5$,
 the degree of $a$ is $4$, and $b$ is not a neighbor of $a$. It  may not be immediately obvious why we can choose such an $a$ and $b$ for the degree sequence 
 $(5^2,4^5,3^4)$; however, if $b$ is a neighbor to all the vertices of degree $4$, then the two vertices of degree $5$ are not neighbors, so we can have a $(9,11)$ 
 graph with the deletion of the two degree five vertices.  As discussed above, this leads to a triangle in $G$.
 
So, we can delete vertices $a$ and $b$ that are not adjacent and of degree four and five. This means that $G-{a,b}$ is a $(9,12)$ graph, so $\chi(G-{a,b}) = -3$. 
 If $G-{a,b}$ is nonplanar and disconnected, then it has either a $K_5$ minor or a $K_{3,3}$ minor with an additional component of order at most three. Whether this 
 component is a tree or a cycle does not matter since either way it will imply a triangle in $G$. So, we'll assume that $G-{a,b}$ is connected.
 
 Denote $G-{a,b}$ as $G^*$. Since $G^*$ is connected and $\chi(G^*) = -3$, if it is nonplanar then it has a $K_{3,3}$ minor, and hence, by Lemma~\ref{lemsplitk33}, $G^*$ is a
 split $K_{3,3}$. Using Lemma~\ref{lemda4} and the restriction that $G$ has only $11$ vertices, we see that $G^*+a$ simplifies to one of the graphs I, II, or III in
 Figure~\ref{figda4}. Notice that II automatically implies a triangle in $G$.  If $G^*$ simplifies to III, then deleting $v_1$ and $w_1$, $v_2$ and $w_2$, or 
 $v_3$ and $w_3$ respectively, shows us that $b$ has $y$, $x$, and $z$ as neighbors as $G$ is not $2$-apex. Since $b$ is of degree five and does not have $a$ as a neighbor,
 then adding it back in will create a triangle in $G$.
 
 If $G^*+a$ simplifies to I we notice that there must another vertex split, $z$, on one of the edges on the triangle formed by $a$, $v_3$, and $w_3$. If $z$ 
 is between $v_3$ and $w_3$ then $G-{v_3,w_3}$ is planar. Having $z$ between $a$ and $v_3$ or $a$ and $w_3$ are symmetric cases, so we will assume 
 $z$ is between $a$ and $w_3$. Since $b$ is adjacent to $z$, if $b$ has $w_3$ as a neighbor there is a triangle. If not, since any four of the other seven possible neighbors of
 $b$ will include at least two neighboring vertices, $G$ will have a triangle.
 
 We conclude that if $G$ cannot have $(5^2,4^5,3^4)$ or $(5,4^7,3^3)$ as its degree sequence.
 
Case $4$: $(4^9,3^2)$

This degree sequence can be considered the hard case for $(11,21)$ graphs since the maximum number of edges we can take away with the deletion of two vertices is $8$. 
In that case, $\Gab$ has $9$ vertices and $13$ edges and there are many such nonplanar graphs. So we will apply a slightly different method for this case.
Assuming that $G$ has the degree sequence $(4^9,3^2)$ we first notice that together, the vertices of degree three have at most six neighbors. Hence, there is a vertex of degree four,
denote it by $v$, whose neighbors are all vertices of degree four. If any of the neighbors of $v$ are mutually adjacent, then $G$ has a triangle. Deleting all four neighbors of $v$
gives us a $(7,5)$ graph, $G^*$, that has at least one vertex of degree zero. Also, since $G$ has maximum degree four then $G^*$ also has maximum degree four. Since 
$\chi(G^*)$ = 2 and $G^*$ has at least one vertex of degree zero, then $G^*$ is one of the following graphs with a degree zero vertex added to it: 
one of the four trees of order five and maximum degree four, a cycle of order five together with a vertex of degree zero, a cycle of order four with a vertex split of degree one 
and a vertex of degree zero, or a cycle of order four together with a tree of order two. Since a cycle of order three is a triangle, we exclude those cases. 
The remaining graphs can be seen in Figure~\ref{fig75}. Our goal is to show that in each case we can add back two of the four vertices we deleted back while maintaining planarity.

\begin{figure}[ht]
\begin{center}

\includegraphics[scale=.75]{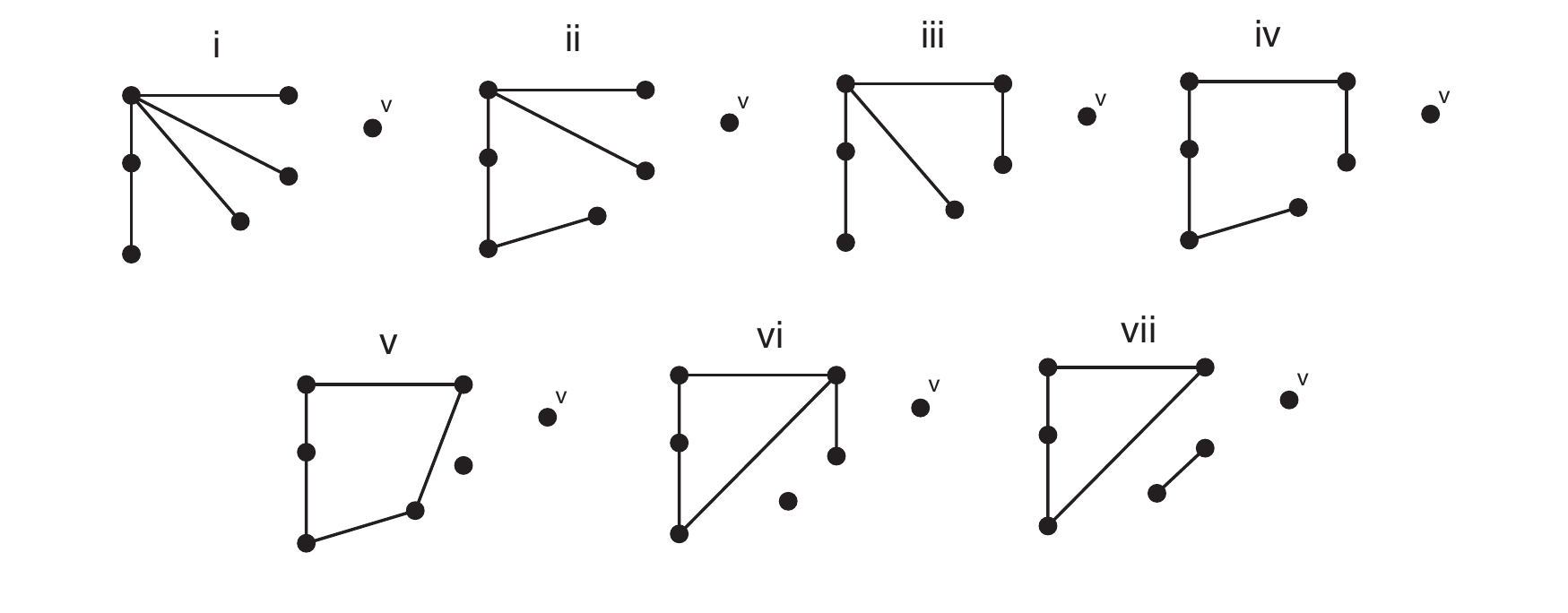}
\caption{The seven triangle--free $(7,5)$ graphs with at least one degree zero vertex and a maximum degree of four.}\label{fig75}
\end{center}
\end{figure}

Since the vertices we delete from $G$ to make $G^*$ all have $v$ as a neighbor, each one will be a vertex of degree three on the graph $G^*-v$. Hence adding one of 
these vertices, call it $a$, back keeps the planarity of $G^*$. Moreover, $G^*+a$ can be arranged such that at most one vertex of $G^*-v$ is not on the outer face and such a vertex, call it 
$u$, will have a degree of two in $G^*+a$. We have three more vertices from which to choose. If all were neighbors of $u$, then $u$ would have a degree of 
five in $G$, which contradicts our degree sequence assumption. So we will be able to add two vertices back into our graph $G^*$ while keeping its planarity. Hence if $G$ has the
degree sequence $(4^9,3^2)$ and does not contain a triangle, then it will be $2$-apex, contradicting Lemma~\ref{lem2ap}.

As we have encountered a contradiction for all possible degree sequences, this concludes the proof.
\end{proof}

\section{10 vertex graphs}

\begin{prop} The only MMIK $(10,21)$ graphs  are the KS graphs, $E_{10}$, $F_{10}$, and $H_{10}$.
\end{prop}

\begin{proof}

Suppose $G$ is a MMIK $(10,21)$ graph. Since $G$ is MMIK it has
minimum degree at least three. As in Remark~\ref{rmkTfree}, we assume $G$ is triangle--free and 
look for a contradiction.

Suppose we can delete two vertices from $G$, $a$ and $b$, such that $\|G-{a,b}\| \leq 10$. If $G-{a,b}$ is non-planar it has either a $K_{3,3}$ minor of a $K_5$ minor.
Since $\chi(G-{a,b})\leq-2$ and $\delta(G-{a,b}) \geq 1$, then $G-{a,b}$ will have a tree of order two or three and by Lemma~\ref{lem2comp}, there will be a triangle in $G$, a contradiction.

If there are $a,b$ $\epsilon$ $V(G)$, such that $\|G-{a,b}\| \leq 11$, then 
$G$ is one of the eleven non-planar graphs in Figure \ref{figNP811}.
With the exception of $iii$ in Figure \ref{figNP811}, which has a tree of order two as a component leading to a triangle in $G$, each graph is a split  $K_{3,3}$. 
Moreover, in each of these graphs, there are two adjacent original vertices
whose neighborhoods are completely comprised of original vertices. Hence, by Lemma~\ref{lemapaths}, adding $a$ (or $b$) back into $G-{a,b}$ will either result in a $1$--apex
graph or will create a triangle. So, $G$ is either 2-apex or has a triangle, a contradiction in either case. 

\begin{figure}[ht]
\begin{center}
\includegraphics[]{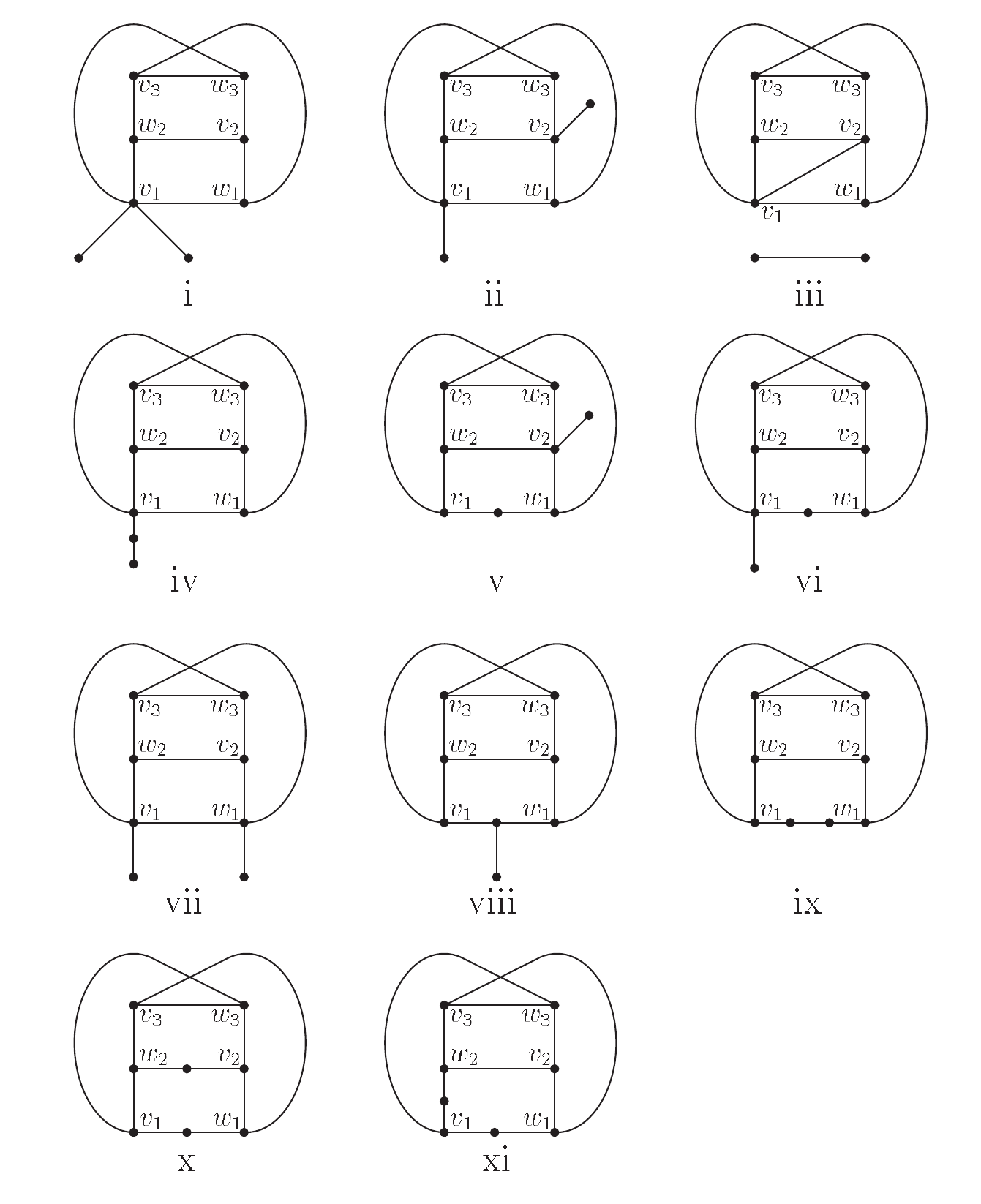}
\caption{Non-planar graphs with eight vertices and eleven edges.}\label{figNP811}
\end{center}
\end{figure}

Thus, it will be enough to consider cases where, for any  $a,b$ $\epsilon$ $V(G)$, $\|G-{a,b}\| \geq 12$.
With this constraint, the only possible degree sequences are $(5^6, 3^4)$, $(5^5, 4^2, 3^3)$, $(5^4, 4^4, 3^2)$, $(5^3, 4^6, 3)$, and $(5^2, 4^8)$. For the first four sequences, 
we realize that if all vertices of degree five are mutually adjacent, then we have a triangle. If not, then 
removing $a$ and $b$ of degree five and non-adjacent give $\| \Gab \| = 11$, a case we considered above. This leaves only the $(5^2, 4^8)$ sequence.

Assume that $G$ has the sequence $(5^2, 4^8)$ and recall that $G$ contains no triangles. 
If the two degree five vertices (call them $a$ and $b$) 
are not neighbors, then $\|G-{a,b}\| = 11$, so we will assume that $a$ and $b$ are neighbors.
This means that $a$ and $b$ do not have any common neighbors, as otherwise, there would be a triangle. Hence, $G-{a,b}$ has the degree sequence $(3^8)$
and, because $G$ is triangle--free, $\Gab$ is a bipartite graph with one part comprised of the neighbors of $a$ in $G$ and the other comprised of the neighbors of $b$ in $G$.
There is only one $3$-regular bipartite graph with two parts of four vertices each. To see this, note that the ``bipartite" complement (i.e. the edges of $K_{4,4}$ not present in $\Gab$) is the disjoint union of four $K_2$'s. Thus, $\Gab$ is the cube and has a planar representation.
Hence, $G$ is $2$-apex, a contradiction.

We conclude that the their are no triangle--free MMIK graphs with $21$ edges and $10$ vertices.
Hence, the $(10,21)$ MMIK graphs are the IK Heawood graphs
of order ten, $E_{10}$, $F_{10}$, and $H_{10}$.
\end{proof}

\section{Graphs that are not $2$--apex.
\label{secN2A}%
}

In this section we prove Propositions~\ref{prop1} and \ref{prop2}.

\begin{proof} (of Proposition ~\ref{prop1})
Since a graph of 20 or fewer edges is $2$--apex~\cite{Ma}, the only N2A graph with $|G| \leq 7$ is $K_7$, which has no 
degree three vertices. So, the proposition is vacuously true for graphs of order seven or less. 

Suppose $G$ is N2A with $|G| = 8$. As discussed in~\cite{Ma}, $G$ must be IK and we refer to the classification
of such graphs due independently to~\cite{CMOPRW} and \cite{BBFFHL}. There are 23 IK graphs on eight vertices, but only four have a vertex 
of degree three. In each case, a $\YT$ move on that vertex results in $K_7$, which is also N2A. 

Again, graphs of size 20 or smaller are $2$--apex. So, we can 
assume $\|G\| = 21$ and $|G| \geq 9$. If $G$ is of order nine and N2A, then, by \cite[Proposition~1.6]{Ma}, $G$ is a Heawood graph (possibly with the addition
of one or two isolated vertices). A $\YT$ move results in the Heawood graph $H_8$ or $K_7 \sqcup K_1$, both of which are N2A.

This leaves the case where $|G| = 10$. Assume $G$ is a $(10,21)$ N2A graph that admits a $\YT$ move to $G'$. For a contradiction, suppose $G'$ is $2$--apex with 
vertices $a$ and $b$ so that $G' - a,b$ is planar. Let $v_0$ be the degree three vertex in $G$ at the center of the $\YT$ move and $v_1,v_2,v_3$ the vertices of the resultant triangle in $G'$. Since $G$ is N2A, it must be
that $\{v_1, v_2, v_3 \}$ is disjoint from $\{a,b\}$. Fix a planar representation 
of $G'-a,b$. The triangle $v_1v_2v_3$ divides the plane into two regions. Let $H_1$ be the induced subgraph on the vertices interior to the triangle and $H_2$ that of the vertices exterior. Then $|H_1|+|H_2| = 4$. Since $G$ is N2A, there is an obstruction to converting the planar representation of $G'-a,b$ into 
a planar representation of $\Gab$. This means that both $H_1$ and $H_2$ contain vertices adjacent 
to each of the triangle vertices $\{v_1, v_2, v_3 \}$. In particular, $H_1$ and $H_2$ each have at least 
one vertex. 

Suppose $|H_1| = |H_2| = 2$. The graph $G-b,v_1$ is non-planar, but, its subgraph $G-a,b,v_1$ is essentially a subgraph of $G'-a,b$ (with the addition of a degree two vertex $v_0$ on the edge $v_2v_3$) and we will use the same planar representation for $G-a,b,v_1$ that we have for $G'-a,b$. Since $G-b,v_1$ is not
planar, there's an obstruction to placing $a$ in the same plane. If we imagine $a$ outside of a disk 
that covers
$G-a,b,v_1$, we see that their is some vertex in an $H_i$ that is hidden from $a$. Without loss of generality, it's one of the vertices $c_1$ or $d_1$ of $H_1$, say $c_1$ that is inaccessible. This means we can assume  that $c_1 v_2 d_1 v_3 $ is a $4$--cycle in $G$. However, as $G'-a,b$ is planar $c_1$ is also hidden from $v_1$ and $c_1v_1$ is not an edge of the graph. 

A similar argument using $G-b,v_2$ allows us to deduce a $4$--cycle $c_2 v_1d_2 v_3$ using the 
vertices $c_2$ and $d_2$ of $H_2$ while showing $c_2 v_2 \not\in E(G)$. However,
it follows that $G - b,v_3$ is planar, a contradiction.

So, we can assume $|H_1| = 3$ while $H_2$ consists of the vertex $c_2$ with $\{v_1, v_2, v_3 \} \subset N(c_2)$. Suppose $H_1$ also has  a vertex, $c_1$, that is adjacent to all three triangle vertices. As $G-b,v_1$ is non-planar, there's a vertex of $H_1$, call it $d_1$, that is hidden from $a$ such that $c_1 v_2 d_1 v_3$ is 
a cycle in $G$ and $d_1v_1 \not\in E(G)$. Similarly, $G-b,v_2$ shows that $c_1 v_1 e_1 v_3$ is in $G$ and $e_1v_2$ is not, $e_1$ being the third vertex of $H_1$. Now, $G-b,v_3$ will be planar unless $d_1e_1 \in V(G)$. However, contracting $d_1e_1$ shows that $G'-a,b$ has a $K_{3,3}$ minor and is non-planar, a contradiction.

If $H_1$ has no vertex $c_1$ that, on its own, is adjacent to the three triangle vertices, then either $H_1$ is
connected, or else it is not but has an edge $c_1d_1$ such that $\{v_1, v_2, v_3 \} \subset N(c_1) \cup N(d_1)$. But, in this latter case, we can rearrange the planar representation of $G'-a,b$ such that the third vertex of $H_1$ is 
exterior to the triangle, returning to the earlier case where $|H_1| = |H_2|=2$. So we will assume $H_1$ 
is connected.

Suppose $H_1$ is not complete, having only two edges $c_1 d_1$ and $d_1 e_1$. Again $G-b,v_1$ shows
that at least two vertices of $H_1$ are in $N(v_2) \cap N(v_3)$ and there are two cases depending on
whether or not $\{c_1, e_1 \} \subset N(v_2) \cap N(v_3)$. If both $c_1$ and $e_1$ are in the intersection, then we can assume $c_1$ is hidden from $a$, meaning $ac_1 \in E(G)$, but $c_1 v_1 \not\in E(G)$. 
Then $G - b, v_2$ shows that $d_1 v_1 e_1 v_3$ is in $G$ and $e_1 v_2$ is not. But then $G-b,v_3$ is planar, a contradiction. If $c_1$ and $e_1$ are not both in $N(v_2) \cap N(v_3)$, we can assume that
$c_1$ and $d_1$ are the common vertices with at most one of those adjacent to $v_1$. If $c_1v_1 \not\in
E(G)$, the argument is the same as above. So, we can assume it's $d_1$ that's hidden, meaning 
$ad_1$ is an edge and $d_1v_1$ is not. In this case, $G- b,v_2$ must be planar, a contradiction.

Finally, if $H_1 = K_3$, then a similar sequence of arguments shows that, in $G'$, the vertices of $H_1$ 
have neighborhoods as follows: $N(c_1) = \{ a,b,d_1, e_1,v_2,v_3 \}$, $N(d_1) = \{a,b,c_1, e_1, v_1, v_3 \}$, 
and $N(e_1) = \{a,b,c_1, d_1, v_1, v_2 \}$. By counting edges, we see that, in fact, $a$ and $b$ each have degree three and we have accounted for all edges in $G'$. Applying the $\TY$ move to recover $G$, we observe that $G$ is $2$--apex (for example, $G - c_1, d_1$ is planar), a contradiction.

We've shown that assuming $G'$ is $2$--apex leads to a contradiction. Thus, the proposition 
also holds in the case $|G| = 10$, which complete the proof.
\end{proof}

\begin{proof} (of Proposition ~\ref{prop2})
Suppose $G$ is MMN2A and $\| G \| = 21$. Note that $\delta(G) \geq 3$ as otherwise a vertex deletion or
edge contraction on a small degree vertex gives a proper minor that is also N2A. This implies $|G| \leq 14$ and
the case of $|G| = 14$ is Proposition~\ref{prop14N2A}. The cases where $|G| \leq 9$ are treated in
\cite{Ma}: a graph with $|G| \leq 8$ is N2A iff it is MMIK, so the proposition follows from the classification
of MMIK graphs of order at most eight; and a graph with $|G| = 9$ is MMN2A if and only if it is one of the Heawood graphs
$E_9$, $F_9$, or $H_9$.

This leaves the case where $|G| = 10$.  If $G$ has a degree three vertex, then apply a $YT$ 
move at that vertex to get a graph $G'$.
By Proposition~\ref{prop1} and the result of \cite{Ma} for graphs of order nine, $G'$ is Heawood, whence $G$ is too.
So, we can assume $\delta(G) \geq 4$ which means the degree sequence of $G$ is $\{4^8,5^2\}$ or $\{4^9,6 \}$.

Suppose there are vertices $a$ and $b$ such that $\| \Gab \| =11$. Then $\Gab$ is one of the graphs of Figure~\ref{figNP811}. 
Since $\delta(G) = 4$, then $\delta(\Gab) \geq 2$. so $\Gab$ is one of graphs ix, x, and xi in the figure. In all three cases,
both $a$ and $b$ must be adjacent to both $v_3$ and $w_3$. For if, for example, $a$ and $v_3$ are not adjacent, then
$G - b, w_3$ is planar. This means $v_3$ and $w_3$ have degree five in $G$, which contradicts the two given degree
sequences for $G$. We conclude there is no choice $a$ and $b$ such that $\| \Gab\| = 11$.

This means $G$ must have degree sequence $\{4^8, 5^2\}$ with the two vertices of degree five adjacent and
$\Gab$ a $(8,12)$ graph. There are two cases depending on whether or not $a$ and $b$ have a common neighbor
in $G$. Suppose first that $c$ is adjacent to both $a$ and $b$. In $\Gab$ vertex $c$ will have degree two and we can use D2 to delete
$c$, arriving either at a $(7,11)$ graph or else a multigraph with a doubled edge. Removing the extra edge if needed, 
let $H$ denote the resulting $(7,11)$ or $(7,10)$ graph.  

\begin{figure}[ht]
\begin{center}
\includegraphics[scale=.5]{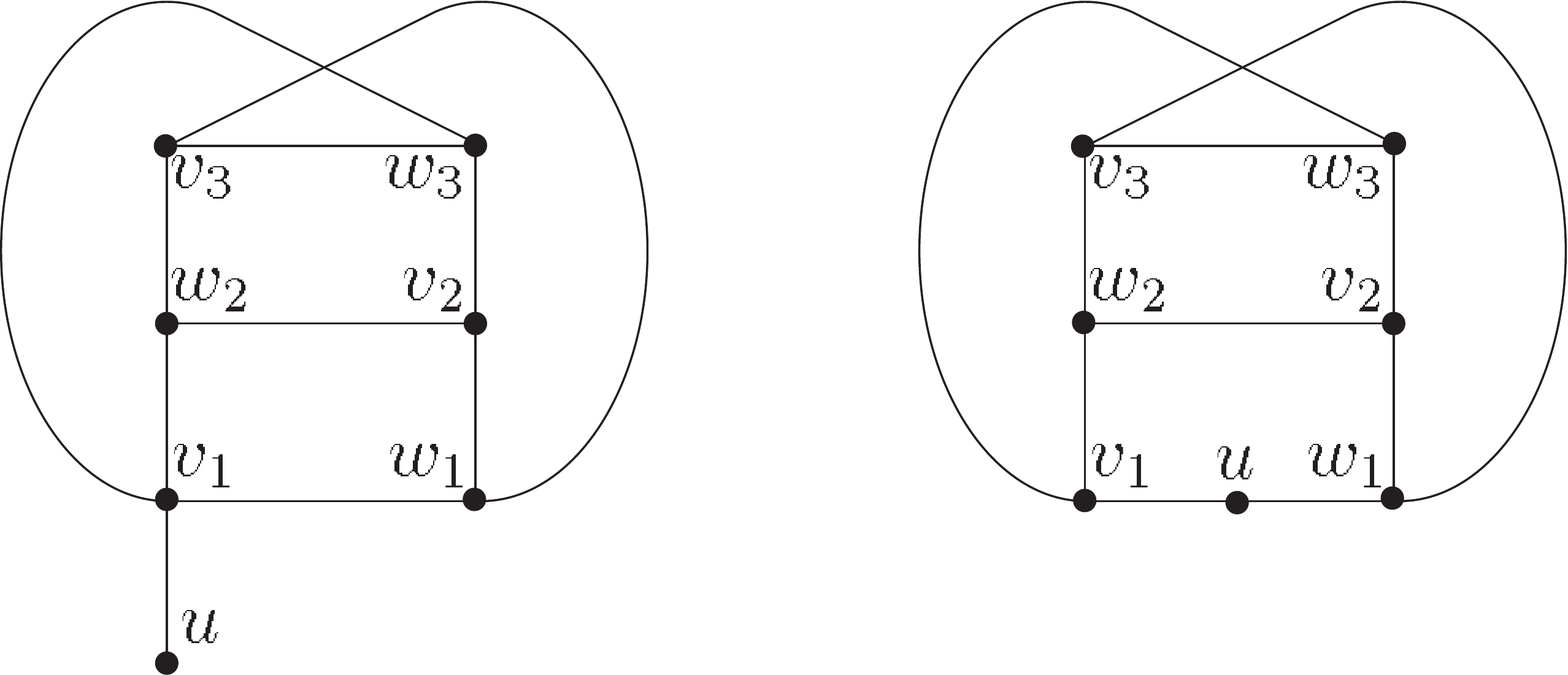}
\caption{The two non-planar (7,10) graphs of minimal degree at least one.}\label{figNP710}
\end{center}
\end{figure}

If $H$ is $(7,10)$, it is one of the two graphs of Figure~\ref{figNP710}. 
In the case of the graph on the left, the doubled edge must be that incident on the degree one vertex as $\delta(\Gab) \geq 2$.
But then the vertex labelled $v_1$ in the figure will have degree five in $\Gab$, contradicting our assumption 
that $a$ and $b$ were the only vertices of degree greater than four. So, we can assume $H$ is the graph to the right in the figure. 
Up to symmetry, the doubled edge of $H$ is either $uv_1$, $v_1w_2$, or $v_2w_2$. We'll examine the first case; the others are similar. 
Doubling $uv_1$ and adding back $c$ leaves $v_1$ of degree four in $\Gab$. Then $G - a,b,v_1$ simplifies to $K_{3,3} - v_1$. Since
$w_1$, $w_2$,  and $w_3$ all have degree three in $\Gab$, they each have exactly one of $a$ and $b$ as a neighbor in $G$. Suppose
$a$ is adjacent to $w_2$. Then $G - a, v_1$ is planar, contradicting $G$ being N2A. For the other two choices of edge doubling, once can again
delete a resulting degree four vertex along with $a$ or $b$ to achieve a planar graph. So $H$ being $(7,10)$ leads to a contradiction.

\begin{figure}[ht]
\begin{center}
\includegraphics[scale=.65]{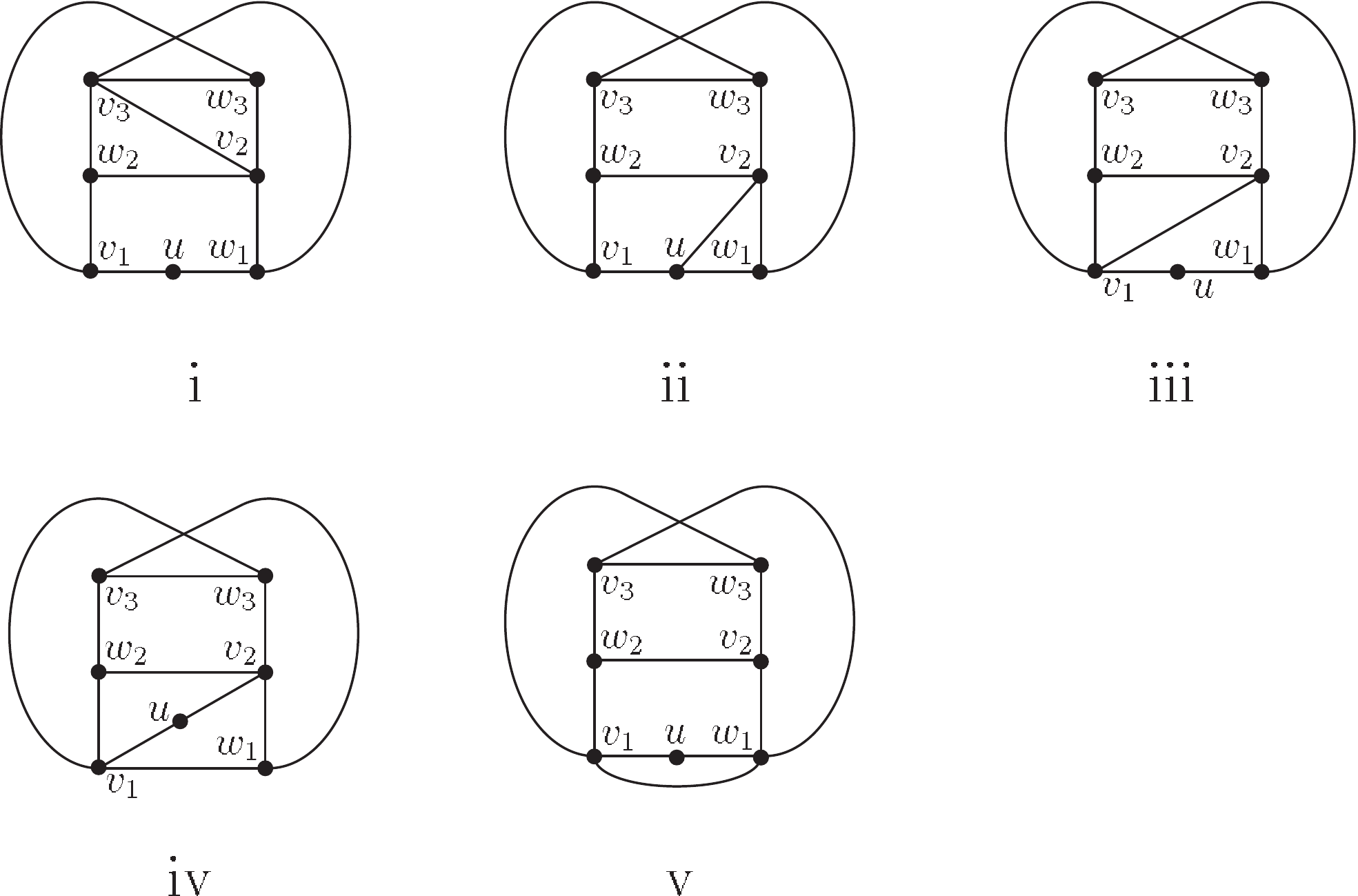}
\caption{The five non-planar (7,11) graphs of minimal degree at least two.}\label{figNP711}
\end{center}
\end{figure}

If $H$ is $(7,11)$, then $\delta(H) = \delta(\Gab) \geq 2$ and $H$ is one of the five graphs of Figure~\ref{figNP711}.
Here we use a similar approach. Deleting one of the degree four vertices of $H$, call it $x$, results in a graph 
$G - a,b,x$ that simplifies to $K_{3,3} - v_1$. Since each of the degree three vertices of $H$ is adjacent to exactly
one of $a$ and $b$, there will be an appropriate choice from those two, say $a$, such that $G - a,x$ is planar, which is a contradiction.
So, $H$ being $(7,11)$ is not possible and we conclude that there is no such vertex $c$ that is adjacent to both $a$ and $b$.

\begin{figure}[ht]
\begin{center}
\includegraphics[scale=.6]{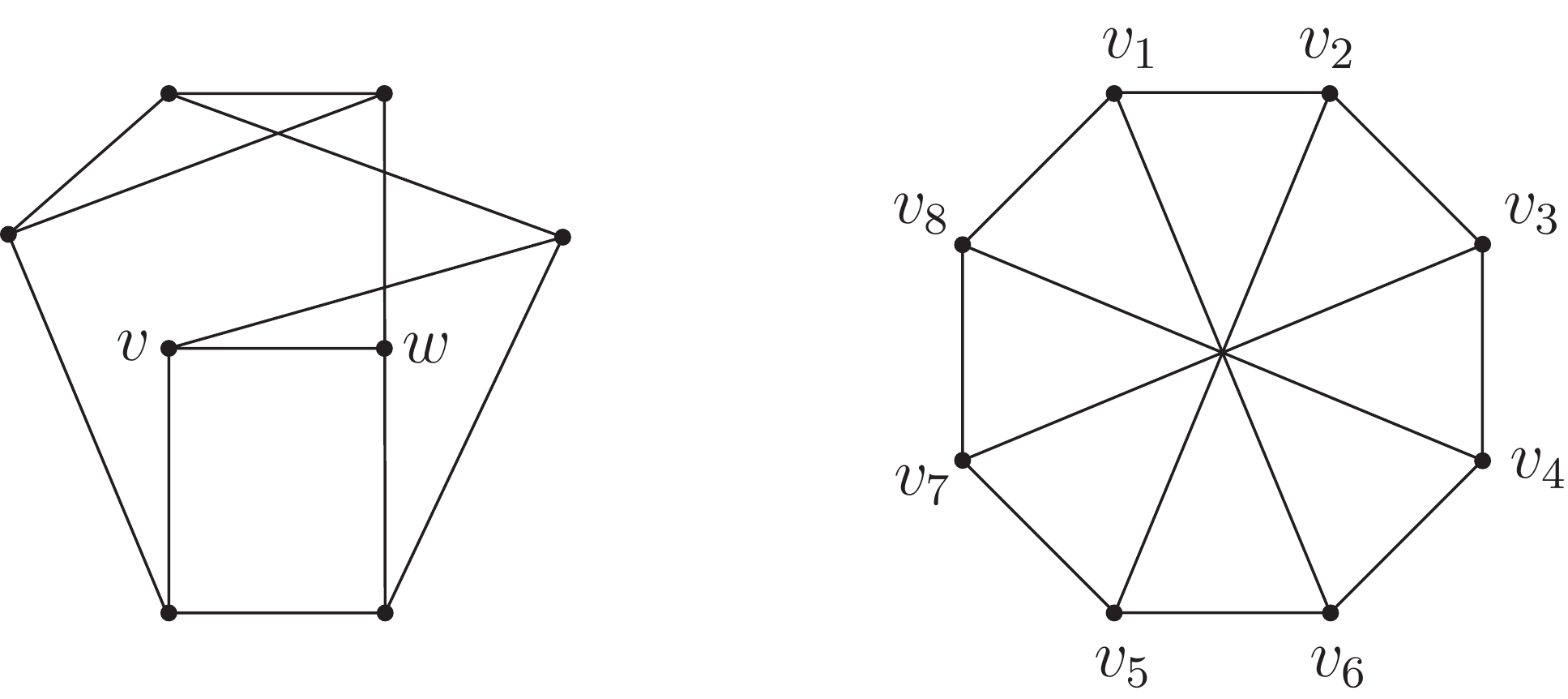}
\caption{The two non-planar cubic graphs of order eight}\label{fig8cubic}
\end{center}
\end{figure}

This means that $\Gab$ is a non-planar cubic graph (i.e., $3$-regular) on eight vertices. There are two such graphs, shown in Figure~\ref{fig8cubic}. If $\Gab$ is the graph to the left in Figure~\ref{fig8cubic}, note that the vertex labelled $v$ is adjacent to exactly one of $a$ and $b$, say $a$. Then $G - a,w$ is planar.

Finally, assume that $\Gab$ is the graph to the right in Figure~\ref{fig8cubic}.
Note that each vertex of $\Gab$ is adjacent to exactly one of $a$ and $b$ in $G$.
If $a$ and $b$ are adjacent to alternate vertices in the $8$--cycle (for example if $\{ v_1,v_3,v_5,v_7 \} \subset N(a)$ and $\{ v_2,v_4,v_6,v_8 \} \subset N(b)$), we obtain graph 20 of figure~\ref{figHea}, a Heawood graph. If not, then we must have two consecutive vertices, say $v_1$ and $v_2$ that share the same neighbor in $\{ a,b \}$, say $a$. That is, we can assume $av_1, av_2 \in E(G)$.  Then $G - a, v_3$ is planar, contradicting $G$ being N2A.

In summary, if $G$ of order 10 is N2A with $\delta(G) > 3$, it must be graph 20 of the Heawood family. This completes the proof of Proposition~\ref{prop2}.
\end{proof}

\section*{Acknowledgments}

We thank Crystal Lane and Anthony Nanfito for participating in a seminar related to this work.
This research was supported in part by a Provost's Research and Creativity Award and a Faculty
Development Award from CSU, Chico.

\end{document}